\UseRawInputEncoding
\documentclass[12pt]{article}

\usepackage{dsfont}
\usepackage{amsfonts,amsmath,amsthm}
\usepackage{amssymb}
\usepackage{algorithm}
\usepackage{algpseudocode}
\usepackage{xcolor}
\usepackage{graphicx}
\usepackage{stmaryrd}        
\usepackage{multirow}

\usepackage[paper=a4paper,dvips,top=2cm,left=2cm,right=2cm,
foot=1cm,bottom=4cm]{geometry}

\newcommand{\vh}{{\bf h}}

\newcommand{\supp}{\text{supp}}

\begin{document}
	\large
	
	\title{Quasi-Irreducibility of Nonnegative Biquadratic Tensors}
\author{Liqun Qi\footnote{Jiangsu Provincial Scientific Research Center of Applied Mathematics, Nanjing, Jiangsu, China.  Department of Applied Mathematics, The Hong Kong Polytechnic University, Hung Hom, Kowloon, Hong Kong.
		({\tt maqilq@polyu.edu.hk}).}
	\and  	Chunfeng Cui\footnote{LMIB of the Ministry of Education, School of Mathematical Sciences, Beihang University, Beijing 100191 China.
		({\tt chunfengcui@buaa.edu.cn}).}
	\and {and \
		Yi Xu\footnote{Southeast University, Nanjing, Jiangsu, China. Nanjing Center for Applied Mathematics, Nanjing, Jiangsu, China. Jiangsu Provincial Scientific Research Center of Applied Mathematics, Nanjing, Jiangsu, China.({\tt yi.xu1983@hotmail.com})}
	}
}
\date{\today}
\maketitle

\begin{abstract}

While the adjacency tensor of a bipartite 2-graph is a nonnegative biquadratic tensor, it is inherently reducible. To address this limitation, we introduce the concept of  quasi-irreducibility in this paper.   The adjacency tensor of a bipartite 2-graph is quasi-irreducible if that bipartite 2-graph is not bi-separable.  This new concept reveals important spectral properties:
although all   M$^+$-eigenvalues   are M$^{++}$-eigenvalues for irreducible nonnegative biquadratic tensors, the  M$^+$-eigenvalues of a quasi-irreducible nonnegative biquadratic tensor can be either M$^0$-eigenvalues or M$^{++}$-eigenvalues.
Furthermore, we establish a max-min theorem for the M-spectral radius of a nonnegative biquadratic tensor.

\medskip



\end{abstract}

\renewcommand{\Re}{\mathds{R}}
\newcommand{\rank}{\mathrm{rank}}
\newcommand{\X}{\mathcal{X}}
\newcommand{\A}{\mathcal{A}}
\newcommand{\I}{\mathcal{I}}
\newcommand{\B}{\mathcal{B}}
\newcommand{\PP}{\mathcal{P}}
\newcommand{\C}{\mathcal{C}}
\newcommand{\D}{\mathcal{D}}
\newcommand{\LL}{\mathcal{L}}
\newcommand{\OO}{\mathcal{O}}
\newcommand{\e}{\mathbf{e}}
\newcommand{\0}{\mathbf{0}}
\newcommand{\1}{\mathbf{1}}
\newcommand{\dd}{\mathbf{d}}
\newcommand{\ii}{\mathbf{i}}
\newcommand{\jj}{\mathbf{j}}
\newcommand{\kk}{\mathbf{k}}
\newcommand{\va}{\mathbf{a}}
\newcommand{\vb}{\mathbf{b}}
\newcommand{\vc}{\mathbf{c}}
\newcommand{\vq}{\mathbf{q}}
\newcommand{\vg}{\mathbf{g}}
\newcommand{\pr}{\vec{r}}
\newcommand{\pc}{\vec{c}}
\newcommand{\ps}{\vec{s}}
\newcommand{\pt}{\vec{t}}
\newcommand{\pu}{\vec{u}}
\newcommand{\pv}{\vec{v}}
\newcommand{\pn}{\vec{n}}
\newcommand{\pp}{\vec{p}}
\newcommand{\pq}{\vec{q}}
\newcommand{\pl}{\vec{l}}
\newcommand{\vt}{\rm{vec}}
\newcommand{\x}{\mathbf{x}}
\newcommand{\vx}{\mathbf{x}}
\newcommand{\vy}{\mathbf{y}}
\newcommand{\vu}{\mathbf{u}}
\newcommand{\vv}{\mathbf{v}}
\newcommand{\y}{\mathbf{y}}
\newcommand{\vz}{\mathbf{z}}
\newcommand{\T}{\top}
\newcommand{\R}{\mathcal{R}}
\newcommand{\Q}{\mathcal{Q}}

\newtheorem{Thm}{Theorem}[section]
\newtheorem{Def}[Thm]{Definition}
\newtheorem{Ass}[Thm]{Assumption}
\newtheorem{Lem}[Thm]{Lemma}
\newtheorem{Prop}[Thm]{Proposition}
\newtheorem{Cor}[Thm]{Corollary}
\newtheorem{example}[Thm]{Example}
\newtheorem{remark}[Thm]{Remark}

\section{Introduction}

Very recently, Cui and Qi \cite{CQ25} studied {the} spectral properties of nonnegative biquadratic tensors.   They {showed} that a nonnegative biquadratic tensor has at least one M$^+$-eigenvalue, i.e., an M-eigenvalue with a pair of nonnegative M-eigenvectors.  The largest M-eigenvalue of a nonnegative biquadratic tensor is an M$^+$-eigenvalue.  It is also the M-spectral radius of that tensor.  All the {M$^+$}-eigenvalues of an irreducible nonnegative biquadratic tensor are M$^{++}$-eigenvalues, i.e., M-eigenvalues with positive M-eigenvector pairs.   For an irreducible nonnegative biquadratic tensor, the largest M$^+$-eigenvalue has a max-min characterization, while the smallest M$^+$-eigenvalue has a min-max characterization.
A Collatz algorithm for computing the largest   M$^+$-eigenvalues was  proposed  and numerical results were reported in \cite{CQ25}.  These results enriched the theories of nonnegative tensors.

Irreducibility plays an important role in the {theoretical analysis of} nonnegative {matrices and tensors}.  {However,} for nonnegative tensors arising from spectral hypergraph theory, irreducibility is too strong \cite{QL17}.   Then weak irreducibility introduced by Friedland, Gaubert and Han \cite{FGH13} {was adopted as an alternative.}
However, we found that the direct extension of weak irreducibility to biquadratic tensors is {too} weak {to produce useful results}.
Consequently, in this paper, we propose the {concept} quasi-irreducibility. Our definition is motivated by bipartite 2-graphs and the concepts of $x$- and $y$-reducibility proposed in \cite{QC25}.

In the next section, we {present} some preliminary knowledge for this paper.   In Section 3, we study bipartite 2-graphs.   We introduce quasi-irreducibility and M$^0$-eigenvalues of nonnegative biquadratic tensors in Section 4.  We show there that the adjacency tensors of bipartite 2-graphs {that} are not bi-separable  are quasi-irreducible.  We further prove that M$^+$-eigenvalues of quasi-irreducible nonnegative biquadratic tensors are either M$^0$-eigenvalues or M$^{++}$-eigenvalues.

In Section 5, we establish a max-min theorem for the M-spectral radius of a nonnegative biquadratic tensor.  In Section 6, we discuss the problem for computing the largest M-eigenvalue of a nonnegative biquadratic tensor.

\section{Preliminaries}

Let $m$ and $n$ be integers greater than $1$. {Denote $[n] := \{ 1, {\dots,} n \}$.}  A real fourth order  
tensor $\A = (a_{i_1j_1i_2j_2}) \in \Re^{m \times n \times m \times n}$ is said to be a  biquadratic tensor.   If for $i_1, i_2 \in [m]$ and $j_1, j_2\in [n]$,
$$a_{i_1j_1i_2j_2} = a_{i_2j_2i_1j_1},$$
then  $\A$ is said to be weakly symmetric.  If furthermore for $i_1, i_2 \in [m]$ and $j_1, j_2\in [n]$,
$$a_{i_1j_1i_2j_2} = a_{i_2j_1i_1j_2}{=a_{i_1j_2i_2j_1}},$$
then $\A$ is said to be symmetric.
Let $BQ(m, n)$ denote   the set of all biquadratic tensors in $\Re^{m \times n \times m \times n}$.
Furthermore, let $NBQ(m,n)$ denote the set of nonnegative $(m\times n\times m\times n)$-dimensional biquadratic tensors.

A {biquadratic} tensor {$\A\in BQ(m,n)$} is  said to be positive semi-definite if for any $\x \in \Re^m$ and $\y \in \Re^n$,
\begin{equation}\label{equ:PSD}
{f(\x, \y) \equiv} \langle \A, \x \circ \y \circ \x \circ \y \rangle \equiv \sum_{i_1, i_2 =1}^m \sum_{j_1, j_2 = 1}^n a_{i_1j_1i_2j_2}x_{i_1}y_{j_1}x_{i_2}y_{j_2} \ge 0,
\end{equation}
and it is said to be  positive definite if for any $\x \in \Re^m, \x^\top \x = 1$ and $\y \in \Re^n, \y^\top \y = 1$,
$${f(\x, \y) > 0.}$$
The {biquadratic} tensor $\A$ is called an SOS (sum-of-squares) biquadratic tensor if $f(\x,\y)$ can be written as a sum of squares.

In 2009, Qi, Dai and Han \cite{QDH09} introduced M-eigenvalues and M-eigenvectors for symmetric biquadratic tensors during their investigation of strong ellipticity condition of the {elastic} tensor in solid mechanics.
Recently, Qi and Cui \cite{QC25} generalized M-eigenvalues to  general (nonsymmetric) biquadratic tensors.  This {definition} was motivated by the study of covariance tensors in statistics \cite{CHHS25}, which are not symmetric {(only weakly symmetric)}, {yet remains} positive semi-definite.

Suppose that $\A = (a_{i_1j_1i_2j_2}) \in BQ(m, n)$.   A   {real} number $\lambda$ is said to be an  {M-eigenvalue} of $\A$ if there are  {real} vectors  $\x = (x_1, {\dots,} x_m)^\top \in {\Re}^m, \y = (y_1, {\dots,} y_n)^\top \in {\Re}^n$ such that the following equations are satisfied:
For {any} $i {\in [m]}$,
\begin{equation} \label{e5}
\sum_{i_1=1}^m \sum_{j_1, j_2=1}^n a_{i_1j_1ij_2}x_{i_1}y_{j_1}y_{j_2} +	\sum_{i_2=1}^m \sum_{j_1, j_2=1}^n a_{ij_1i_2j_2}y_{j_1}x_{i_2}y_{j_2} = 2\lambda x_i;
\end{equation}
for any $j {\in [n]}$,
\begin{equation} \label{e6}
\sum_{i_1,i_2=1}^m\sum_{j_1=1}^n a_{i_1j_1i_2j}x_{i_1}y_{j_1}x_{i_2} +	\sum_{i_1,i_2=1}^m\sum_{j_2=1}^n a_{i_1ji_2j_2}x_{i_1}x_{i_2}y_{j_2} = 2\lambda y_j;
\end{equation}
and
\begin{equation} \label{e7}
\x^\top \x ={ \y^\top \y} = 1.
\end{equation}
Then $\x$ and $\y$ are called the corresponding  {M-eigenvectors}.
We may   rewrite equations \eqref{e5} and \eqref{e6} as $\frac12\A\cdot\vy\vx\vy+\frac12\A\vx\vy\cdot\vy=\lambda \vx$ and $\frac12\A\vx\cdot\vx\vy+\frac12\A\vx\vy\vx\cdot=\lambda\vy$,  respectively.

The following theorem was established in \cite{QC25}.

\begin{Thm} \label{T2.1}
Suppose that $\A \in BQ(m, n)$.  Then $\A$ always {has} M-eigenvalues.  Furthermore, $\A$ is positive semi-definite if and only if all of its M-eigenvalues are nonnegative, {and} $\A$ is positive definite if and only if all of its M-eigenvalues are positive.
\end{Thm}

Let $\A \in BQ(m, n)$
{and $\lambda_{\max}(\A)$ be the largest M-eigenvalue of $\A$.  Then we have
\begin{equation} \label{lamax}
	\lambda_{\max}(\A) = \max \{ f(\x, \y) : \x^\top \x = \y^\top \y = 1, \x \in \Re^m, \y \in \Re^n \}.
\end{equation}
}
{Denote by} $\rho_M(\A)$  the M-spectral radius of $\A$, i.e., the largest absolute value among all M-eigenvalues of $\A$.

{Suppose} that $\lambda$ is an M-eigenvalue of $\A$ {associated} with a pair of nonnegative M-eigenvectors $\x \in \Re_+^m$ and $\y \in \Re_+^n$.   Then $\lambda$ is also nonnegative, i.e., $\lambda \ge 0$.    We call $\lambda$ {an} M$^+$-eigenvalue of $\A$. Furthermore, if both $\x$ and $\y$ are positive, we call $\lambda$ an M$^{++}$-eigenvalue of $\A$.

The following theorems, which form the weak Perron-Frobenius theorem of irreducible nonnegative biquadratic tensors,  were established in \cite{CQ25}.

\begin{Thm}\label{Thm:rho=lmd_max}
Let $\A = \left( a_{i_1j_1i_2j_2}\right) \in NBQ(m, n)$, where $m, n \ge 2$.  Then we have
\begin{equation} \label{rholamax}
\rho_M(\A) = \lambda_{\max}(\A),
\end{equation}
and $\lambda_{\max}(\A)$ is an M$^+$-eigenvalue of $\A$. {Consequently,   $\A$ has at least one  M$^+$-eigenvalue.}
\end{Thm}	

\begin{Thm}\label{Thm:eigpair_positive}
Suppose that $\A = \left(a_{i_1j_1i_2j_2}\right) \in NBQ(m, n)$ is irreducible, where $m,n\ge 2$.  
{Assume that} $\lambda$  is  an M$^+$ eigenvalue of $\A$  with nonnegative M-eigenvector  $\{\bar \vx, \bar \vy\}$.    Then $\lambda$ is a positive M$^{++}$-eigenvalue.
{Consequently,   $\A$ has at least one  M$^{++}$-eigenvalue.}
\end{Thm}

\section{Bipartite $2$-Graphs}

A bipartite hypergraph $G = (S, T, E)$ has two vertex sets $S = \{ u_1, u_2, \dots, u_m \}$, $T = \{ v_1, v_2, \dots, v_n \}$ and an edge set $E = \{ e_1, e_2, \dots, e_p \}$.   An edge $e_l = (s_l, t_l)$ of $G$ consists of a subset $s_l \subset S$ and a subset $t_l \subset T$.  Assume that there are no two edges with the same subset pair of $S$ and $T$.  Bipartite hypergraphs are useful in the study of uniform hypergraphs \cite{CD12, HQX15,QL17}.

Suppose that we have a bipartite hypergraph $G = (S, T, E)$.  If for all edge $e_l = (s_l, t_l), l \in [p]$, $s_l$ and $t_l$ have the same cardinality $k$, then $G$ is called a bipartite uniform hypergraph or a bipartite $k$-graph.  In particular, a bipartite $1$-graph is simply called a bipartite graph, which has been studied extensively.  We now study bipartite $2$-graphs.

Suppose that we have a bipartite $2$-graph $G = (S, T, E)$.
We may express $G$ by a biquadratic tensor $\A = \left(a_{i_1j_1i_2j_2}\right){\in NBQ(m,n)}$ as follows:
\begin{equation}\label{equ:adj}
a_{i_1j_1i_2j_2} =\left\{\begin{array}{cl}
1, & \text{ if }e_{i_1j_1i_2j_2} = (s_{i_1i_2}, t_{j_1j_2}) \in E;\\
0, & \text{ otherwise},
\end{array}\right.
\end{equation}	
where $s_{i_1i_2} = (i_1, i_2)$ and $t_{j_1j_2} = (j_1, j_2)$.
Then $\A$ is a symmetric nonnegative biquadratic tensor.
Given a nonnegative valued function $\varphi: E \rightarrow \Re_+$, we may also define the  adjacency tensor $\A=(a_{i_1j_1i_2j_2}){\in NBQ(m,n)}$ of a weighted graph $G=(S,T,E,\varphi)$ as follows,
\begin{equation}\label{equ:adj_phi}
a_{i_1j_1i_2j_2} =\left\{\begin{array}{cl}
\varphi(e_{i_1j_1i_2j_2}), & \text{ if } e_{i_1j_1i_2j_2} = (s_{i_1i_2}, t_{j_1j_2}) \in E;\\
0, & \text{ otherwise},
\end{array}\right.
\end{equation}	
{where $s_{i_1i_2} = (i_1, i_2)$ and $t_{j_1j_2} = (j_1, j_2)$.}
Then $\A$ is still a nonnegative biquadratic tensor.

By \cite{CQ25}, a biquadratic tensor $\A = \left(a_{i_1j_1i_2j_2}\right) \in BQ(m, n)$ is reducible {if}
either  it is $x$-reducible, i.e., there {is a} nonempty proper index subset $J_x{\subsetneq}  [m]$ and a  proper index {$j\in[n]$} 
such that
\begin{equation}\label{equ:x_reducible}
{a_{i_2ji_1j}+	a_{i_1ji_2j}}=0, \ \forall i_1\in J_x, \forall i_2\notin J_x,
\end{equation}
or it is $y$-reducible, i.e., there is a  proper index {$i\in[m]$}
and a nonempty  proper index subset  $J_y{\subsetneq}  [n]$ such that
\begin{equation}\label{equ:y_reducible}
{a_{ij_1ij_2}+a_{ij_2ij_1}}=0, \ \forall j_1\in J_y, \forall j_2\notin J_y.
\end{equation}
Then, by the definition of bipartite $2$-graphs, their adjacency tensors are always reducible. Actually, according to our definition, for the entries of the adjacency tensor $\A = \left(a_{i_1j_1i_2j_2}\right)$ of a bipartite $2$-graph $G$, we always have $a_{i_1j_1i_2j_2} = 0$ if either $i_1 = i_2$ or $j_1=j_2$.  Therefore, the adjacency tensor $\A$ is always both $x$-reducible and $y$-reducible.  Consequently, as in the nonnegative cubic tensor case, we have to consider weaker conditions.

Given a bipartite $2$-graph  $G = (S, T, E,\varphi)$, we may also define 	its signless Laplacian biquadratic tensor {as follows,}
\begin{equation}\label{def:Q}
\Q=\D^0+\D^x+\D^y+\A \in NBQ(m,n),
\end{equation}
where $D^0=(d^0_{i_1j_1i_2j_2}){\in NBQ(m,n)}$ is a   diagonal biquadratic tensor with diagonal elements
\begin{equation}\label{def:d^0}
d^0_{i_1j_1i_2j_2} = 	 \left\{\begin{array}{cl}
\sum\limits_{i_2'=1}^m\sum\limits_{j_2'=1}^n    a_{i_1j_1 i_2'j_2'}, & \text{ if } i_1=i_2 \text{ and } j_1=j_2;\\
0, & \text{ otherwise},
\end{array}\right.
\end{equation}
and	 $D^x=(d^x_{i_1j_1i_2j_2}){\in NBQ(m,n)}$ and $D^y=(d^y_{i_1j_1i_2j_2}){\in NBQ(m,n)}$ are defined by
\begin{equation}\label{def:d^x}
d^x_{i_1j_1i_2j_2} = \left\{\begin{array}{cl}
\sum\limits_{i_2'=1}^m   a_{i_1j_1i_2'j_2}, & \text{ if }i_1=i_2;\\
0, & \text{ otherwise},
\end{array}\right.
\end{equation}
\begin{equation}\label{def:d^y}
d^y_{i_1j_1i_2j_2} = \left\{\begin{array}{cl}
\sum\limits_{j_2'=1}^n    a_{i_1j_1 i_2j_2'}, & \text{ if } j_1=j_2;\\
0, & \text{ otherwise},
\end{array}\right.
\end{equation}
for all $i,i_1,i_2\in[m]$ and $j,j_1,j_2\in[n]$,   respectively.

Similarly, we can define its Laplacian biquadratic tensor {as follows,}
\begin{equation}\label{def:L}
\LL=\D^0-\D^x-\D^y{+\A \in BQ(m,n).}
\end{equation}
Then we have the following results.
\begin{Lem}
Given a weighted bipartite 2-graph $G = (S, T, E,\varphi)$,  both	the signless Laplacian biquadratic tensor defined by \eqref{def:Q} and the  Laplacian biquadratic tensor defined by \eqref{def:L} are positive semi-definite and SOS.
\end{Lem}
\begin{proof}
For any $\vx\in \Re^m$ and $\vy\in \Re^n$, we have
\begin{align*}
\Q\vx\vy\vx\vy  = \sum_{i_1,i_2=1}^m \sum_{j_1,j_2=1}^n a_{i_1j_1i_2j_2} (x_{i_1}+x_{i_2})^2(y_{j_1}+y_{j_2})^2,
\end{align*}
and
\begin{align*}
\LL\vx\vy\vx\vy  = \sum_{i_1,i_2=1}^m \sum_{j_1,j_2=1}^n a_{i_1j_1i_2j_2} (x_{i_1}-x_{i_2})^2(y_{j_1}-y_{j_2})^2.
\end{align*}
{This completes the proof. }
\end{proof}

{Let us further examine} the application backgrounds of bipartite $2$-graphs $G = (S, T, E)$.  We may think that $S$ is a set of persons, and $T$ is a set of tools.   Suppose that a working group requires exactly two persons $i_1$ and $i_2$, and two tools $j_1$ and $j_2$. While a working group consists of two persons and two tools, there exist compatibility constraints that prevent arbitrary pairs of persons and tools from forming valid groups. If two person $i_1, i_2$ and two tools $j_1$ and $j_2$ can form a working group, then we say that edge $(i_1, i_2, j_1, j_2)$ is in $E$.   We say that $S$ is $T$-separable if there exist a proper partition $(S_1, S_2)$ of $S$, and two vertices $j_1, j_2 \in T$ such that {for all} $i_1 \in S_1$ and $i_2 \in S_2$, $(i_1, i_2, j_1, j_2) \not \in E$.  Similarly, we say that $T$ is $S$-separable if there exist a proper partition $(T_1, T_2)$ of $T$, and two vertices $i_1, i_2 \in S$ such that {for all}  $j_1 \in T_1$ and $j_2 \in T_2$, $(i_1, i_2, j_1, j_2) \not \in E$.   If $S$ is $T$-separable, or $T$ is $S$-separable, then we say that $G$ is bi-separable.

\section{Quasi-Irreducible Nonnegative Biquadratic Tensors}

Let $\A = \left(a_{i_1j_1i_2j_2}\right) \in NBQ(m,n)$.
We say that $\A$ is  quasi-reducible {if it is}
either   $x$-quasi-reducible, i.e., there are a nonempty proper index subset $J_x{\subsetneq}  [m]$ and two distinct indices {$j_1, j_2 \in[n]$, $j_1 \not = j_2$} 
such that
\begin{equation}\label{equ:weak-x_reducible}
{a_{i_2j_1i_1j_2}+	a_{i_1j_1i_2j_2}}=0, \ \forall i_1\in J_x, \forall i_2\notin J_x.
\end{equation}
or it is $y$-quasi-reducible, i.e., there are  two distinct indices {$i_1, i_2 \in[m]$, $i_1 \not = i_2$} 
and a nonempty  proper index subset  $J_y{\subsetneq}  [n]$ such that
\begin{equation}\label{equ:weak-y_reducible}
{a_{i_1j_1i_2j_2}+a_{i_1j_2i_2j_1}}=0, \ \forall j_1\in J_y, \forall j_2\notin J_y.
\end{equation}
We say that $\A$ is quasi-irreducible if it is not quasi-reducible.

In fact, if $j_1=j_2$ in equation~(\ref{equ:weak-x_reducible}), then $x$-quasi-reducible becomes the  $x$-reducible in \cite{CQ25}; Similarly, if $i_1=i_2$ in equation~(\ref{equ:weak-y_reducible}), then $y$-quasi-reducible  becomes   the $y$-reducible in \cite{CQ25}.
Take $m=n=2$ as an example. Then $\A\in NBQ(2,2)$ is  $x$-reducible if
\begin{equation*}
a_{1121} + a_{2111} >0 \text{ and } a_{1222} + a_{2212} >0,
\end{equation*}
while $\A$ is  $x$-quasi-reducible if
\begin{equation*}
a_{1122} + a_{2112} >0 \text{ and } a_{1221} + a_{2211} >0.
\end{equation*}
Thus, these two definitions are fundamentally distinct, with neither being a special case of the other.
Our definition of quasi-irreducibility naturally encompasses hyperedges in {bipartite 2-graphs,} which in fact motivated our formulation.
{We illustrate the distinction between irreducible and quasi-irreducible nonnegative biquadratic tensors in Fig.~\ref{fig:BQreducible}.}
\begin{figure}
\begin{center}
\includegraphics[width=0.8\linewidth]{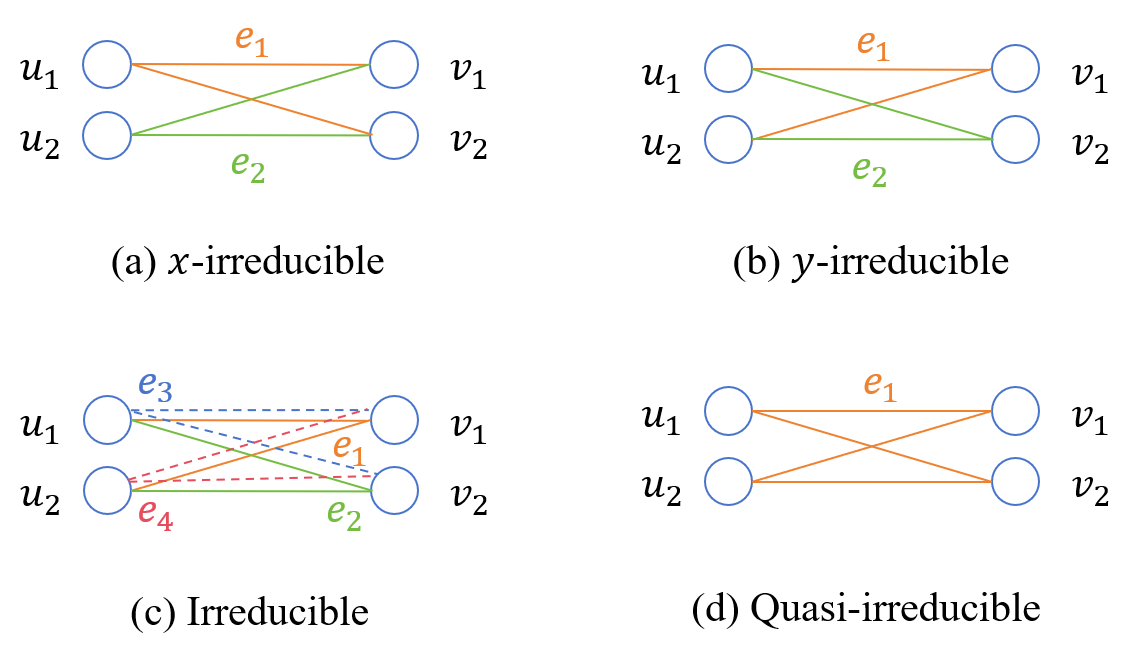}
\end{center}
\caption{{Illustration of $x$-irreducible, $y$-irreducible,      irreducible, and quasi-irreducible (i.e., both $x$- and $y$-quasi-irreducible) nonnegative biquadratic tensors for $m=n=2$.}}\label{fig:BQreducible}
\end{figure}


We have the following  {propositions.}

\begin{Prop}\label{prop:bi-sep_quasi-irr}
Suppose that {a} bipartite $2$-graph $G = (S, T, E)$ is not bi-separable, where $|S| \ge 2$ and $|T| \ge 2$, and $\A$ is the adjacency tensor of $G$.  Then $\A$ is a quasi-irreducible biquadratic tensor.
\end{Prop}
\begin{proof}   By the definition given in \eqref{equ:adj}  in Section 2, {we have} $\A \in NBQ(m, n)$, where $m, n \ge 2$.  Suppose that $\A$ is $x$-quasi-reducible.  Then there are two indices $i_1, i_2 \in [m]$, such that $T$ is $S$-separable.  This contradicts our assumption on $G$.  Thus, $\A$ is  $x$-quasi-irreducible.   Similarly, we may show that $\A$ is  $y$-quasi-irreducible.   Hence, $\A$ is  quasi-irreducible.
\end{proof}

\begin{Prop}
Suppose that the bipartite $2$-graph $G = (S, T, E)$ is not bi-separable, where $|S| \ge 2$ and $|T| \ge 2$, and $\Q$ is the signless Laplacian tensor of $G$.  Then we have the following conclusions:
\begin{itemize}
\item[(i)] $\D^x$ is $x$-irreducible, but $y$-reducible and quasi-reducible;
\item[(ii)] $\D^y$ is $y$-irreducible, but $x$-reducible and quasi-reducible;
\item[(iii)] $\D^0$ is $x$- and $y$-reducible,   and $x$- and $y$-quasi-reducible;
\item[(iv)] $\Q$ is irreducible and quasi-irreducible.
\end{itemize}
\end{Prop}
\begin{proof}
It follows from Proposition~\ref{prop:bi-sep_quasi-irr}  that $\A$ is  quasi-reducible. Consequently, for any nonempty proper index subset $J_x{\subsetneq}  [m]$ and two distinct indices {$j_1, j_2 \in[n]$, $j_1 \not = j_2$}, we have
\begin{equation}\label{equ:weak-x_irreducible}
{a_{i_2j_1i_1j_2}+	a_{i_1j_1i_2j_2}}>0, \ \exists i_1\in J_x, \exists i_2\notin J_x,
\end{equation}
and for any   two distinct indices {$i_1, i_2 \in[m]$, $i_1 \not = i_2$}  and  nonempty  proper index subset  $J_y{\subsetneq}  [n]$, we have
\begin{equation}\label{equ:weak-y_irreducible}
{a_{i_1j_1i_2j_2}+a_{i_1j_2i_2j_1}}>0, \ \exists j_1\in J_y, \exists j_2\notin J_y.
\end{equation}

(i) By equation~\eqref{equ:weak-y_irreducible}, for any proper index  $i\in[m]$ and nonempty proper index subset $J_y{\subsetneq}  [n]$, there exist  $j_1\in J_y$ and  $j_2\notin J_y$ such that
\begin{equation*}
d^x_{ij_1ij_2}+	d^x_{ij_2ij_1}=\sum_{i_2=1}^m a_{ij_1i_2j_2}+	a_{ij_2i_2j_1} \ge a_{ij_1i_2j_2}+	a_{ij_2i_2j_1} > 0, \forall i_2\neq i.
\end{equation*}
This shows that $\D^x$ is $x$-irreducible. Furthermore, since $d^{x}_{i_1j_1i_2j_2}=0$ for any $j_1=j_2$ and $i_1\neq i_2$, we have  $\D^x$  is $y$-reducible and quasi-reducible, respectively.

(ii) By equation~\eqref{equ:weak-x_irreducible}, for any proper index  $j\in[n]$ and nonempty proper index subset $J_x{\subsetneq}  [m]$, there exist  $i_1\in J_x$ and  $i_2\notin J_x$ such that
\begin{equation*}
d^y_{i_1ji_2j}+	d^y_{i_2ji_1j} =\sum_{j_2=1}^m a_{i_1ji_2j_2}+a_{i_2ji_1j_2} \ge 		a_{i_1ji_2j_2}+a_{i_2ji_1j_2}> 0, \forall j_2\neq j.
\end{equation*}
This shows that $\D^y$ is $y$-irreducible. Furthermore, since $d^{x}_{i_1j_1i_2j_2}=0$ for any $i_1=i_2$ and $j_1\neq j_2$, we have  $\D^y$  is $x$-reducible and quasi-reducible, respectively.

(iii) It follows from $\D^0$ is a diagonal biquadratic tensor.

(iv) It follows directly from  $\Q=\A+\D^0+\D^x+\D^y$, $\A$ is quasi-irreducible,    $\D^x$ is $x$-irreducible, and $\D^y$ is $y$-irreducible.

This completes the proof.
\end{proof}

Suppose that $\lambda$ is an M-eigenvalue of $\A$. If $\A$ has a pair of M-eigenvectors $\x \in \Re^m$ and $\y \in \Re^n$, such that either $|\supp(\x)|=1$ or $|\supp(\y)| =1$, then we call $\lambda$ an $M^0$-eigenvalue of $\A$.
{Now we are ready to present the main result of this section.}
\begin{Thm} \label{Thm:QNBQ}
Suppose that $\A = \left(a_{i_1j_1i_2j_2}\right) \in NBQ(m,n)$ is quasi-irreducible, and $\lambda$ is an $M^+$ eigenvalue of $\A$.  Then $\lambda$ is either an $M^0$ eigenvalue or an $M^{++}$ eigenvalue of $\A$.
\end{Thm}
\begin{proof}
By the definition of M-eigenvalues and M-eigenvectors, we have
$${\A\bar\vx\bar \vy \cdot\bar \vy +}\A\cdot\bar \vy \bar \vx\bar \vy={2}\lambda \bar\vx \text{ and }\A\bar \vx\bar\vy \bar \vx\cdot +\A\bar \vx\cdot \bar \vx\bar \vy={2}\lambda \bar\vy.$$
By the properties of nonnegative biquadratic tensors, $\lambda$ is an $M^+$ eigenvalue of $\A$.   Assume that $\lambda$ is not an $M^0$ eigenvalue.

We now prove that $\lambda$ is an $M^{++}$ eigenvalue by showing that $\bar \vx> \0_m$ and $\bar \vy > \0_n$ utilizing the method of contradiction.

Suppose that $\bar\vx \ngtr \0_m$. Let $J_x=[m]\setminus \supp(\bar\vx)$. Then $J_x$ is a nonempty set. For any $i\in J_x$, we have
\begin{eqnarray*}
& & {\sum_{i_1=1}^m \sum_{j_1,j_2=1}^n a_{i_1j_1ij_2} \bar x_{i_1}\bar y_{j_1}\bar y_{j_2} +} \sum_{i_2=1}^m \sum_{j_1,j_2=1}^n a_{ij_1i_2j_2} \bar x_{i_2}\bar y_{j_1}\bar y_{j_2} \\
&=& {\sum_{i_1\notin J_x} \sum_{j_1,j_2\in\supp(\bar\vy)} a_{i_1j_1ij_2} \bar x_{i_1}\bar y_{j_1}\bar y_{j_2} +}\sum_{i_2\notin J_x} \sum_{j_1,j_2\in\supp(\bar\vy)} a_{ij_1i_2j_2} \bar x_{i_2}\bar y_{j_1}\bar y_{j_2}\\
&=& 0.
\end{eqnarray*}
This shows $a_{ij_1i_2j_2}=0$ for all  $i\in J_x$,  $i_2\notin J_x$, and $j_1,j_2\in\supp(\bar\vy)$.  As $\lambda$ is not an M$^0$ eigenvalue, we have $\left|\supp(\bar\vy)\right| \ge 2$.  This leads to a contradiction with the assumption that $\A$ is quasi-irreducible. Hence $\bar \vx> \0_m$. Similarly, we could show that $\bar \vy > \0_n$. Therefore, we have $\lambda > 0$ and is an M$^{++}$-eigenvalue.
\end{proof}

{The above theorem, together with Theorem \ref{Thm:rho=lmd_max}, forms the weak Perron-Frobenius theorem of quasi-irreducible nonnegative biquadratic tensors.}

%
%
%

The next example shows that a quasi-irreducible nonnegative biquadratic tensor may have no M$^{++}$ eigenvalue.

\begin{example}\label{ex4.4}
Let $\A\in NBQ(2,2)$ be defined by $a_{1111}=1$, $a_{2222}=2$, $a_{1212}=3$,  $a_{1122} = a_{1221} =a_{2112}=a_{2211}=1$ and all other elements are zeros. Then we may verify that $\A$ is reducible, but quasi-reducible.
Furthermore, the M$^+$-eigenvalues of $\A$ are
\[ 0,\ 1.0000,\     2.0000,\   3.0000,\]
and the corresponding eigenvectors are
\begin{eqnarray*}
\left\{\begin{array}{c}
	\vx = (0,\ 1)^\top,\\
	\vy = (1,\    0)^\top.
\end{array}\right.
& \left\{\begin{array}{c}
	\vx = (1,\    0)^\top,\\
	\vy = (1,\ 0)^\top,\\
\end{array}\right.
\left\{\begin{array}{c}
	\vx = (0,\ 1)^\top,\\
	\vy = (0,\    1)^\top,\\
\end{array}\right.
\left\{\begin{array}{c}
	\vx = (1,\    0)^\top,\\
	\vy = (0,\ 1)^\top.
\end{array}\right.
\end{eqnarray*}
\end{example}

\section{A Max-Min Theorem for Nonnegative Biquadratic Tensors}

Denote
$${S_+^{m-1}}=\{\vx\in\Re_+^m: \sum x_i^2=1\}$$
as the nonnegative section of the unit sphere surface  in the $m$-dimensional space and
$$   S_{++}^{m-1}=\{\vx\in\Re_{++}^m: \sum x_i^2=1\}$$
as the interior set of ${S_+^{m-1}}$.  For any $\vx\in \Re^m$ and $\vy\in \Re^n$, let
\begin{equation}\label{equ:gh}
\vg =\frac12 ( \A \cdot \vy\vx\vy + \A \vx \vy\cdot\vy), \ \ \vh=\frac12 ( \A \vx\cdot\vx\vy + \A \vx  \vy\vx\cdot).
\end{equation}
Let $\A\in NBQ(m,n)$ be  quasi-irreducible. We define the following two functions for all  $\vx \in \Re_+^m\setminus \{\0_m\}$ and $\vy \in \Re_+^n\setminus \{\0_n\}$.
\begin{eqnarray}\label{equ:uv}
v(\vx,\vy) = \min_{\substack{i: x_i>0,\\j: y_j>0}} \left\{\frac{g_i}{x_i}, \frac{h_j}{y_j}\right\}, \ \
u(\vx,\vy) = \max_{\substack{i: x_i>0,\\j: y_j>0}} \left\{\frac{g_i}{x_i}, \frac{h_j}{y_j}\right\}.
\end{eqnarray}
Then $v(\vx,\vy) \le u(\vx,\vy)$.
Here, we require the indices $i,j$ to be subsets of $[m]$ and $[n]$, respectively, to avoid the indeterminate form  $\frac00$. As illustrated  in Example~\ref{ex4.4}, an  M$^+$-eigenvalue may not necessarily be an  $M^{++}$-eigenvalue.
We also define
\begin{equation}\label{equ:rho_*^*}
\rho_* = \inf_{\vx\in   S_{+}^{m-1}, \vy\in   S_{+}^{n-1}} u(\vx,\vy) \text{ and } \rho^* = \sup_{\vx\in   S_{+}^{m-1}, \vy\in   S_{+}^{n-1}} v(\vx,\vy)
\end{equation}

By Theorem~\ref	{Thm:rho=lmd_max},  $\lambda_{\max}(\A)$    is attainable and coincides with  the M-spectral radius $\rho_M(\A)$ of $\A$,
and $\lambda_{\max}(\A)$ is an M$^+$-eigenvalue of $\A$.
The following theorem shows that  the value $\rho^*$  is attainable and is equal to $\rho_M(\A)$.

\begin{Thm} \label{Thm:max-min}
Suppose that $\A = \left(a_{i_1j_1i_2j_2}\right) \in NBQ(m,n)$.  Then we have
\[\rho^*=\lambda_{\max}(\A) =\rho_M(\A).\]
\end{Thm}
\begin{proof}
Let $\bar \vx,\bar \vy$ denote the optimal solutions to  problem~\eqref{lamax}. Then we have $v(\bar \vx,\bar \vy)=\lambda_{\max}(\A)=\rho_M(\A)$, which shows that  $\rho^* \ge \lambda_{\max}(\A)$. Next, we prove that  $\rho^* = \lambda_{\max}(\A)$ using the method of contradiction.

Suppose that $\rho^*> \lambda_{\max}(\A)$.
Then for any $\epsilon>0$, there is  $\tilde\vx\in   S_{+}^{m-1}$ and $\tilde \vy\in   S_{+}^{n-1}$ such that
$v(\tilde \vx,\tilde \vy)\ge \rho^*- \epsilon$.
By choosing $\epsilon = \frac{1}{2}(\rho^*- \lambda_{\max}(\A))$, we have
$$v(\tilde \vx,\tilde \vy)\ge\frac{1}{2}(\rho^*+ \lambda_{\max}(\A)).$$
Therefore, it follows that $\tilde g_i\ge \frac{1}{2}(\rho^*+ \lambda_{\max}(\A)) \tilde x_i^*$ and  $\tilde h_j\ge\frac{1}{2}(\rho^*+ \lambda_{\max}(\A))\tilde y_j^*$ for all $i\in[m]$ and $j\in[n]$. Consequently, we have  $\A\tilde\vx\tilde\vy\tilde\vx\tilde\vy = \tilde\vg^\top \tilde\vx=\tilde\vh^\top \tilde\vy\ge \frac{1}{2}(\rho^*+ \lambda_{\max}(\A))$.
This contradicts   the definition of $\lambda_{\max}$ in \eqref{lamax}.
Thus, we have $\rho^*= \lambda_{\max}(\A)$  and completes the proof.
\end{proof}

\section{The Largest M-eigenvalue of A Nonnegative Biquadratic Tensor}

Suppose that $\A \in BQ(m,n)$.  Then the problem for computing its largest
M-eigenvalue problem is an NP-hard problem.   This was proved in \cite{LNQY10}.   However, when $\A \in NBQ(m,n)$.   This problem is different.   This is just like in the cubic tensor case.   In general, the problem to computing the spectral radius of a high order cubic tensor is an NP-hard problem.  But there are many efficient algorithms to compute the largest eigenvalue of a nonnegative cubic tensor \cite{QL17}.
Thus, we have the following problem.

Problem 1: To find the largest M-eigenvalue of a nonnegative biquadratic tensor $\A \in NBQ(m,n)$.   Theoretically, is this problem polynomial-time solvable?  Practically, are there efficient algorithms to solve this problem?

By Theorem \ref{Thm:rho=lmd_max}, the largest M-eigenvalue of a nonnegative biquadratic tensor is an M$^+$-eigenvalue.    This is the first step for solving this problem.

We now consider a subproblem of Problem 1.

Problem 2: To find the largest M-eigenvalue of an irreducible nonnegative biquadratic tensor $\A \in NBQ(m,n)$.   Theoretically, is this problem polynomial-time solvable?  Practically, are there efficient algorithms to solve this problem?

By Theorem \ref{Thm:eigpair_positive}, the largest M-eigenvalue of an irreducible nonnegative biquadratic tensor is an M$^{++}$-eigenvalue.  A Collatz algorithm was proposed in \cite{CQ25} to solve this problem

By Theorem \ref{Thm:max-min}, can we also construct a Collatz algorithm to solve Problem 1?

This paper raised another subproblem of Problem 1.

Problem 3: To find the largest M-eigenvalue of a quasi-irreducible nonnegative biquadratic tensor $\A \in NBQ(m,n)$.   Theoretically, is this problem polynomial-time solvable?  Practically, are there efficient algorithms to solve this problem?

Theorem \ref{Thm:QNBQ} indicates that for a quasi-irreducible nonnegative biquadratic tensor, the largest M-eigenvalue $\lambda_{\max}(\A)$ is either an M$^0$-eigenvalue or an M$^{++}$-eigenvalue.  Can we utilize this information to solve Problem 3?

	\bigskip	
	
	
	{{\bf Acknowledgment}}
	This work was partially supported by Research  Center for Intelligent Operations Research, The Hong Kong Polytechnic University (4-ZZT8),    the National Natural Science Foundation of China (Nos. 12471282 and 12131004), the R\&D project of Pazhou Lab (Huangpu) (Grant no. 2023K0603),  the Fundamental Research Funds for the Central Universities (Grant No. YWF-22-T-204), and Jiangsu Provincial Scientific Research Center of Applied Mathematics (Grant No. BK20233002).

	{{\bf Data availability} Data will be made available on reasonable request.

		{\bf Conflict of interest} The authors declare no conflict of interest.}

	


\end{document}